\documentclass[12pt]{article}

\usepackage{amsmath,amssymb,amsthm}
\usepackage[breaklinks=true,colorlinks=true]{hyperref}
\usepackage[dvipsnames]{xcolor}
\usepackage[top=.7in, bottom=.9in, left=1.1in, right=1.1in, marginparwidth=1in, marginparsep=0.1in]{geometry}

\numberwithin{equation}{section}
\newtheorem{theorem}{Theorem}[section]

\newtheorem{lemma}[theorem]{Lemma}

\theoremstyle{remark}

\definecolor{darkblue}{rgb}{0,0,0.7}

\usepackage{marginnote} %

\newcommand{\bke}[1]{\left( #1 \right)}

\newcommand{\bket}[1]{\left\{ #1 \right\}}
\newcommand{\norm}[1]{\| #1 \|}

\newcommand{\al}{\alpha}
\newcommand{\be}{\beta}
\newcommand{\de}{\delta}
\newcommand{\e}{\epsilon}
\newcommand{\ga}{{\gamma}}
\newcommand{\Ga}{{\Gamma}}
\newcommand{\la}{\lambda}

\newcommand{\Om}{{\Omega}}

\newcommand{\si}{\sigma}

\newcommand{\De}{\Delta}
\renewcommand{\th}{\theta}

\newcommand{\R}{{\mathbb R }}

\newcommand{\NN}{{\mathbb N}}
\newcommand{\ZZ}{{\mathbb Z}}

\newcommand{\pd}{{\partial}}
\newcommand{\nb}{{\nabla}}
\newcommand{\lec}{\lesssim}

\newcommand{\dis}{\displaystyle}

\renewcommand{\div}{\mathop{\mathrm{div}}}

\newcommand{\Bog}{\mathop{\mathrm{Bog}}\nolimits}

\newcommand{\EQ}[1]{\begin{equation}\begin{split} #1 \end{split}\end{equation}}
\newcommand{\EQN}[1]{\begin{equation*}\begin{split} #1 \end{split}\end{equation*}}
\newcommand{\EQL}[2]{\begin{equation}\label{#1}\begin{split} #2 \end{split}\end{equation}}

\makeatletter
\newcommand{\xRightarrow}[2][]{\ext@arrow 0359\Rightarrowfill@{#1}{#2}}
\makeatother

\newcommand{\loc}{\mathrm{loc}}

\usepackage{accents}

\newcommand{\hA}{\widehat{\mathbf A}} %
\newcommand{\hE}{\widehat{\mathbf E}}
\newcommand{\CB}{C_{\mathrm{bg}}}

\begin{document}
\title{Liouville type theorems for stationary Navier-Stokes equations}
\author{Tai-Peng Tsai%
\thanks{\noindent Department of Mathematics, University of British Columbia, Vancouver, BC V6T 1Z2, Canada. ttsai@math.ubc.ca.
 The work of Tsai was partially supported by NSERC grant RGPIN-2018-04137.}}
\date{Dedicated to Hideo Kozono on the occasion of his 60th birthday}
\maketitle 

\begin{abstract} 
We show that any smooth stationary solution of the 3D incompressible Navier-Stokes equations
in the whole space, the half space, or a periodic slab must vanish under the condition that for some 
$0 \le \delta \le 1<L$ and $q=6(3-\delta)/(6-\delta)$,
$$\liminf_{R \to \infty} \frac 1R \|u\|^{3-\delta}_{L^{q}(R<|x|<LR)}=0.$$ 
We also prove sufficient conditions allowing shrinking radii ratio $L= 1+R^{-\alpha}$.
Similar results hold on a slab with zero boundary condition by assuming stronger decay rates. 
We do not assume global bound of the velocity. 
The key is to estimate the pressure locally in the annuli with radii ratio $L$ arbitrarily close to 1.
\end{abstract}

\section{Introduction}\label{S1}
Consider the Liouville problem of 3D stationary incompressible Navier-Stokes equations
\EQ{\label{SNS}
 - \De u+(u\cdot \nb )u+\nb p=0, \quad \div u=0,  \quad \text{in }\Om, 
}
where the domain $\Om$ is either the whole space $\R^3$, the half space $\R^3_+$ with zero boundary condition, or the slab $\Om=\R^2\times (0,1)$ with zero or periodic boundary condition (BC).
In the classical setting $\Om=\R^3$, one asks if the only $H^1_{\loc}$ solution satisfying
\EQ{\label{D.soln}
\int_\Om |\nb u|^2 < \infty, \quad \lim_{|x|\to \infty} u(x)=0
}
is zero.  A solution satisfying \eqref{D.soln} is called a \emph{$D$-solution}. This problem has been reformulated by Seregin and Sverak%
 to whether the only solutions satisfying
\EQ{
u \in H^1_{\loc}\cap L^\infty(\Om)
}
are constant vectors. The same problems can be posed in other domains, and can be asked in the subclass of axisymmetric flows.

We now review the literature.
In the 2 dimensional case, the problem \eqref{D.soln} in the plane $\R^2$ is solved by Gilbarg and
Weinberger \cite{GW}.
For the 3 dimensional problem, it is not even known if a general
D-solution satisfying \eqref{D.soln} has any explicit decay rate. The following is a list of vanishing results with extra integral or decay
assumptions on the solution. 
Galdi \cite[Theorem X.9.5]{Galdi} proved that if $u$ is a D-solution in $\R^3$ and $u \in L^{9/2}(\R^3)$, then $u = 0$. The same proof works for dimension $n\ge 4$ assuming only \eqref{D.soln} without additional integrability condition.
This result was improved
by a log factor in Chae and Wolf \cite{CW}, assuming
\[
\int_{\R^3}|u|^{\frac 92} \bket{\ln (2+1/|u|)}^{-1}dx<\infty.
\]
In \cite{Chae}, Chae proved that a D-solution with $\De u \in L^{6/5}(\R^3)$ is zero. %
Seregin \cite{Seregin}
proved that a solution in $\R^3$ is 0 if $u \in L^6(\R^3) \cap BMO^{-1}$. Kozono, Terasawa and Wakasugi \cite{KTW} showed that a D-solution $u$ in $\R^3$ is zero if either the vorticity decays like $c|x|^{-5/3}$ at infinity, or $\norm{u}_{L^{9/2,\infty}}\le c$, with $c=\e \norm{\nb u}_2^{2/3}$ and $\e$ a small constant. In \cite{Seregin2,CW2}, the authors prove Liouville type theorems for smooth
solutions $u$ under growth conditions on the $L^s$ mean oscillation over $B_r$ of the \emph{potential tensor} of $u$.
Lin, Uhlmann and Wang \cite{LUW} proved the following lower bound using Carleman estimates for a bounded solution $u$ in $\Om= \R^3 \setminus \overline B_1$: Let $M(r) = \inf _{|x|=r} \int_{B_1(x)} |u|^2$ and $\la = \norm{u}_{W^{1,\infty}(\Om)}$.
Then there exist $C(\la)>0$ and $R_0(\la, M(10))>10$ such that
\EQ{\label{LUWeq2}
M(r) \ge \exp (-C r^2 \log r) , \quad \forall r>R_0.
}
This result does not assume any boundary condition. 
It implies that a 
bounded
solution in an exterior domain in $\R^3$ must be zero if it satisfies
\EQ{\label{LUWeq3}
\liminf_{r \to \infty} \exp (C r^2 \log r)\,M(r)<1, \quad \forall C>0.
}
It is probably the first Liouville type result with a liminf condition.
Let $L^{q,l}$ denote the Lorentz spaces.
Seregin and Wang \cite{SW} prove the vanishing of $u$ assuming either for $3<q<\infty$, $3 \le l \le \infty$ (or $q=l=3$), 
\EQ{\label{SW1}
\liminf_{R\to \infty} R^{\frac 23 - \frac 3q} \norm{u}_{L^{q,l}(B_R \setminus B_{R/2})}\le \e \norm{\nb u}_2^{2/3},
}
with $\e$ a small constant,
or for $12/5<q<3$, $1 \le l \le \infty$, $\ga > \frac 13 + \frac 1q$,
\EQ{\label{SW2}
\liminf_{R\to \infty} R^{\ga - \frac 3q} \norm{u}_{L^{q,l}(B_R \setminus B_{R/2})}=0.
}
They don't assume the solution is globally bounded.
Note that \eqref{SW1} for $q>3$ follows from $q=3$ case as
\EQ{\label{SW3}
R^{ - \frac 13} \norm{u}_{L^{3}(B_R \setminus B_{R/2})} 
\lec R^{\frac 23 - \frac 3q} \norm{u}_{L^{q,l}(B_R \setminus B_{R/2})}.
}

For other domains,
the proof of Galdi \cite[Theorem X.9.5]{Galdi} can be extended to $\R^3_+$ and slabs easily. We are not aware of any other previous results for the half space.
On a slab with zero BC,  Pileckas and Specovius-Neugebauer \cite{Pil,PiSN} studied the asymptotic decay of solutions. They proved, under certain weighted integral assumption on the velocity $u$ and its derivatives with a force in \eqref{SNS}, $u(x)$ decays like $1/|x|$. Then the vanishing of $u$ with zero force follows easily.  This was extended by Carrillo, Pan, Zhang, and Zhao \cite[Theorem 1.1]{CPZZ}, showing that any D-solution satisfying \eqref{D.soln} in a slab with zero BC is zero.

There is also a rich literature on the Liouville problem for the subclass of axisymmetric solutions. As we will not discuss it here, we only refer to \cite{Wang, CPZZ,KTW2} and their references.

The following is our first main result.

\begin{theorem}[Whole space] \label{th1}
Suppose $u\in H^1_{\loc}(\R^3)$ is a weak solution of \eqref{SNS} in $\Om=\R^3$. 

(a) If for some constants $0  \le \de \le 1$ and $L>1$,
\EQ{\label{th1eq1}
\liminf_{R \to \infty} \frac 1R \norm{u}^{3-\de}_{L^{q}(R<|x|<LR)}=0,
\quad q(\de)=\frac{3-\de}{1-\de/6},
}
then $u=0$.
 
(b) If for some constants $0  \le \de \le 1$ and $\al\ge0$,
\EQ{\label{th1eq2}
\liminf_{R \to \infty} R^\be \norm{u}_{L^{q}(R<|x|<R+R^{1-\al})}=0,
\quad q(\de)=\frac{3-\de}{1-\de/6},
}
where $\be=\be(\de,\al)= \max \bket{\frac{\frac {3-\al}{q}-2+3\al}{2-\de},\ \frac {-1+2\al}{3-\de}}$, then $u=0$.
\end{theorem}

Comments on Theorem \ref{th1}:
\begin{enumerate}
\item Part (a) is a borderline improvement of Seregin and Wang \cite{SW}, by allowing
$\ga=\frac13+\frac1q$ in \eqref{SW2}.
Note that $q$ is decreasing in $\de$ with lower bound $q(1)=12/5$, which is allowed in Theorem \ref{th1} but excluded in \eqref{SW2}.

\item
Our proof is different: \cite{SW} is based on a Caccioppoli type inequality for the nonlinear equation, while our proof is based on pressure-independent interior estimates of the Stokes system, see Lemma \ref{ST}.

\item
As in \eqref{LUWeq3}, \eqref{SW1} and \eqref{SW2}, conditions \eqref{th1eq1} and \eqref{th1eq2} use $\liminf$, not limit. 
We do not assume $u \in L^\infty(\Om)$ nor $\nb u \in L^2(\Om)$. The condition $u\in H^1_{\loc}$ implies $u \in C^\infty_\loc$.
Since we do not assume a global bound of $u$, we need to estimate the pressure locally.

\item Unlike \eqref{SW3}, the condition \eqref{th1eq1} for lower $q$ does not follow from itself for higher $q$ by H\"older inequality. For example, \eqref{th1eq1} for $\de=1$ and $q=12/5$ is
\EQ{%
\liminf_{R \to \infty} R^{-1/2} \norm{u}_{L^{\frac{12}5}(R<|x|<LR)}=0.
}
It does not follow from \eqref{th1eq1} for $\de=0$ and $q=3$
\EQ{\label{eq1.10}
\liminf_{R \to \infty} R^{-1/3} \norm{u}_{L^{3}(R<|x|<LR)}=0.
}

 Condition \eqref{eq1.10} implies that, for any nonzero $H^1_\loc$ solution $u$ in $\R^3$ and any $L>1$, there are $\e>0$ and $R_0\gg 1$ such that
\EQ{\label{th1eq3}
 \frac 1R \int_{ R<|x|<LR} |u|^3\, dx\ge \e,\quad \forall R>R_0.
}
We have similar lower bounds for other $q$ from \eqref{th1eq1}.
They are in the spirit of \eqref{LUWeq2}.

\item The main feature of part (b) is that the ratio of the outer and the inner radii is shrinking to 1 when $\al>0$. (It contains part (a) as a special case with $\al=0$.) To be able to prove it, we need explicit bounds of the \emph{Bogovskii map} in such annuli, see Lemma \ref{Bog-annulus}. For the exponent $\be(\de,\al)$
with $0 \le \de \le 1$ and $0 \le\al <\infty$,
 $\be(\de,\al)= \frac {-1+2\al}{3-\de}$ if and only if $5\al \de-24\al -3\de+6\ge 0$, in particular if $ \al  \le 3/19$.

\item When  $\be(\de,\al)= (\frac {3-\al}{q}-2+3\al)/(2-\de)$, one may get alternative conditions as in Theorem \ref{th4} (b), by not applying H\"older inequality to bound $\norm{u}_{q/2}$ by $\norm{u}_q$ in \eqref{eqHolder}.

\end{enumerate}

For the following three theorems,
we denote a point $x \in \Om$ as $x=(x',x_3)$ with $x'\in \R^2$.

\begin{theorem}[Half space] \label{thmB}
Let $\Om=\R^3_+=\{(x',x_3)\in \R^3:\, x_3>0\}$.
Suppose $u\in H^1_{\loc}(\overline {\Om})$ is a weak solution of \eqref{SNS} in $\Om$ with zero boundary condition.

(a) If for some constants $0  \le \de \le 1$ and $L>1$,
\EQ{\label{thmBeq1}
\liminf_{R \to \infty} \frac 1R \norm{u}^{3-\de}_{L^{q}(R<|x|<LR)}=0,
\quad q(\de)=\frac{3-\de}{1-\de/6},
}
then $u=0$.
 
(b) If for some constants $0  \le \de \le 1$ and $\al\ge0$,
\EQ{\label{thmBeq2}
\liminf_{R \to \infty} R^\be \norm{u}_{L^{q}(R<|x|<R+R^{1-\al})}=0,
\quad q(\de)=\frac{3-\de}{1-\de/6},
}
where $\be=\be(\de,\al)= \max \bket{\frac{\frac {3-\al}{q}-2+3\al}{2-\de},\ \frac {-1+2\al}{3-\de}}$, then $u=0$.

\end{theorem}

Comments on Theorem \ref{thmB}:
\begin{enumerate}
\item The statement and proof of Theorem \ref{thmB} are similar to those of Theorem \ref{th1}, but we also need to estimate $\nb u$ and $p$ on the boundary without pressure assumption. 
\item
Conditions \eqref{thmBeq1} and \eqref{thmBeq2} use $\liminf$, not limit. 
We do not assume $u \in L^\infty(\Om)$ nor $\nb u \in L^2(\Om)$. The condition $u\in H^1_{\loc}(\overline \R^3_+)$ implies $u \in C^\infty_\loc(\overline \R^3_+)$, see \cite{Kang}.

\end{enumerate}

\begin{theorem}[Periodic slab] \label{thmC}
Let $\Om = \R^2 \times (\R/\ZZ)$. We denote a point $x \in \Om$ as $x=(x',x_3)$ with $(x',0)=(x',1)$.
Suppose $u\in H^1_{\loc}(\Om)$ is a weak solution of \eqref{SNS} in $\Om$.

(a) If for some constants $0  \le \de \le 1$ and $L>1$,
\EQ{\label{thmCeq1}
\liminf_{R \to \infty} \frac 1R \norm{u}^{3-\de}_{L^{q}(R<|x'|<LR)}=0,
\quad q(\de)=\frac{3-\de}{1-\de/6},
}
then $u=0$.
 
(b) If for some constants $0  \le \de \le 1$ and $\al\ge0$,
\EQ{\label{thmCeq2}
\liminf_{R \to \infty} R^{\tilde\be} \norm{u}_{L^{q}(R<|x'|<R+R^{1-\al})}=0,
\quad q(\de)=\frac{3-\de}{1-\de/6},
}
where $\tilde\be= \be_{ps}(\de,\al)= \max \bket{\frac{\frac {2-\al}{q}-2+3\al}{2-\de},\ \frac {-1+2\al}{3-\de}}$, then $u=0$.
\end{theorem}

Comments on Theorem \ref{thmC}:
\begin{enumerate}
\item In Theorem \ref{thmC}, conditions \eqref{thmCeq1} and \eqref{thmCeq2} use $\liminf$, not limit.  We do not assume $u \in L^\infty(\Om)$ nor $\nb u \in L^2(\Om)$. The condition $u\in H^1_{\loc}(\Om)$ implies $u \in C^\infty_\loc(\Om)$.

\item Note $\be_{ps}(\de,\al)$ differs from $\be(\de,\al)$ in Theorems \ref{th1} and \ref{thmB} in that the numerator $3-\al$ is replaced by $2-\al$.

\end{enumerate}

\begin{theorem}[Zero BC slab] \label{th4}
Let $\Om = \R^2 \times (0,1)$. 
Suppose $u\in H^1_{\loc}(\overline \Om)$ is a weak solution of \eqref{SNS} in $\Om$ with zero boundary condition $u(x',0)=u(x',1)=0$.

(a) If for some constants $0  \le \de \le 1$ and $L>1$,
\EQ{\label{th4eq1}
\liminf_{R \to \infty} R^{2/q} \norm{ u}^{2-\de}_{L^{q}(R<|x'|<LR)}\to 0,
}
where $q=q(\de)=\frac{3-\de}{1-\de/6}$,
 then $u=0$.

(b) If for some constants $6/5  \le r \le 2$ and $L>1$,
\EQ{\label{th4eq1b}
\liminf_{R \to \infty} \int_0^1\int_{ R<|x'|<LR} \bke{|u|^r+ |u|^{2r}}\, dx'\,dx_3=0,
}
 then $u=0$.
\end{theorem}

Comments on Theorem \ref{th4}:
\begin{enumerate}
\item In Theorem \ref{th4} the conditions \eqref{th4eq1} and \eqref{th4eq1b} are $\liminf$, not limit.  We do not assume $u \in L^\infty(\Om)$ nor $\nb u \in L^2(\Om)$. The condition $u\in H^1_{\loc}(\overline \Om)$ implies $u \in C^\infty_\loc(\overline\Om)$;  see \cite{Kang}.

\item  Its proof is different from those for Theorems \ref{th1}-\ref{thmC} as we cannot obtain the local pressure estimate by scaling, and we get an additional $R$ factor. As a result, we have a positive exponent for $R$ in \eqref{th4eq1}. Moreover, we cannot vary the radii ratio $L$.
 
\item A D-solution satisfying \eqref{D.soln} in a slab with zero BC is shown to be zero by \cite[Theorem 1.1]{CPZZ}.   It is extended by Theorem \ref{th4} since \eqref{D.soln} implies \eqref{th4eq1b}: 
By Poincar\'e inequality in $x_3$ direction and zero BC, for $A_R=\{x'\in \R^2: R<|x'|<LR\}$,
\[
\int_0^1 \int _{A_R} |u|^2 dx \le C\int_0^1 \int _{A_R} |\pd_{x_3} u|^2 dx %
\]
which vanishes 
as $R \to \infty$ by \eqref{D.soln}. By regularity theory,  \eqref{D.soln} and zero BC imply $u \in L^\infty(\Om)$.  We have $\int_0^1 \int _{A_R} |u|^4 dx \le \norm{u}_{L^\infty(\Om)}^2 \int_0^1 \int _{A_R} |u|^2 dx=o(1)$.

\item It is possible to prove $u=0$ assuming $\liminf_{R\to \infty}\int_0^1\int_{A_R} (|u|^r + |u|^s)=0$ with $s>2r$, by modifying the proof of part (b). We skip it to keep the presentation simple.
\end{enumerate}

The key to the above theorems is the estimate of the pressure 
\[
\inf_{c\in \R} \norm{p-c}_{L^q(E)}
\]
in an annulus-like region $E$, based on integral bounds of $u$ in a slightly larger region.
Here $E = B_{LR} \setminus B_R$ for $\Om=\R^3$, 
$E = B_{LR}^+ \setminus B_R^+$  for $\Om=\R^3_+$,
$E = (B_{LR}' \setminus B_R')\times (0,R)$ for a periodic slab, and
$E = (B_{LR}' \setminus B_R')\times (0,1)$ for a zero BC slab. 
Here  $B_R$ is the ball in $\R^3$ of radius $R$ centered at the origin,  $B_R^+= B_R\cap \R^3_+$,
while $B_R'$ is a ball in $\R^2$.
These estimates are based on Lemma \ref{p-est} and the estimates of the corresponding Bogovskii maps, Lemmas \ref{Bog-annulus}, \ref{Bog-cylinder} and \ref{CPZZ}. After we prove these lemmas in \S\ref{S2}, we will prove Theorem \ref{th1} in \S\ref{S3}, Theorem \ref{thmB} in \S\ref{S4}, Theorem \ref{thmC} in \S\ref{S5}, and Theorem \ref{th4} in \S\ref{S6}.

\section{Bogovskii map and pressure estimate}
\label{S2}

We first recall the Bogovskii map (see \cite[Lemma III.3.1]{Galdi} and \cite[\S2.8]{Tsai-book}). For a domain $E \subset \R^n$, denote
\[
L^q_0(E)=\{ f \in L^q(E):\ \textstyle {\int_E} f =0\}.
\]

\begin{lemma}\label{Bogovskii} 
Let $E$ be a bounded Lipschitz domain in $\R^n$, $2 \le n \in \NN$. Let $1<q<\infty$. There is a linear map
\[
\Bog: L^q_0(E) \to W^{1,q}_0(E;\R^n),
\]
such that for any $f\in L^q_0(E)$, $ v = \Bog f$ is a vector field that satisfies
\[
v\in W^{1,q}_0(E)^n,\quad
\div v = f, \quad 
\norm{\nb v}_{L^q(E)} \le \CB(E,q) \norm{f}_{L^q(E)},
\]
where the constant $\CB$ does not depend on $f$.
If $RE = \{ Rx: x\in E\}$, then $\CB(RE,q)=\CB(E,q)$.
\end{lemma}

This map is non-unique and we usually fix a choice that almost minimizes the constant $\CB$. Strictly speaking $\CB$ depends on this choice. The last statement $\CB(RE,q)=\CB(E,q)$ is because for given $\Bog$ defined on $E$, we can define $\Bog_R$ on $RE$ as follows: For $\bar f \in L^q_0(RE)$, let $f(x) = \bar f(Rx)$ for $x\in E$, $v=\Bog f \in W^{1,q}_0(E)$, and 
$\bar v = \Bog_R \bar f$ is given by $\bar v(y) = Rv(R^{-1} y)$.

\medskip

The constant $\CB$ appears in the following pressure estimate.

\begin{lemma} \label{p-est} 
Let $E$ be a bounded Lipschitz domain in $\R^n$.  Let $p \in L^q(E)$, 
$1<q<\infty$. Then 
\EQ{\label{pest-eq1}
 \norm{p - (p)_{E}}_{L^q(E)} \le C_0 \sup_{ \zeta \in W^{1,q'}_0(E),\ \norm{\nb \zeta}_{L^{q'}(E)}=1} \int p \div \zeta
}
where $(p)_E = \frac 1{|E|}\int _E p$ and $C_0 = 2 \CB(E,q')$.
\end{lemma}

Eq.~\eqref{pest-eq1} %
is used in \cite{Sverak-Tsai} to prove Lemma \ref{ST} below.
Its proof follows that of \cite[Lemma IV.1.1]{Galdi} although stated differently, and is given here for completeness and to specify the constant.
\begin{proof}
We may replace $p$ by $p-c$ in \eqref{pest-eq1} and hence we may assume $(p)_E=0$. Let $g=|p|^{q-2}p - (|p|^{q-2}p)_E$. Then
\[
\int_Eg=0, \quad \norm{g}_{L^{q'}(E)} \le 2 \norm{p}_{L^q(E)}^{q-1}.
\]
By Lemma \ref{Bogovskii}, there is a solution $w \in W^{1,q'}_0(E)^n$ of
\[
\div w = g, \quad \norm{\nb w}_{L^{q'}(E)} \le C_1\norm{g}_{L^{q'}(E)},
\]
where $C_1=\CB(E,q')$.
Denote $N=\sup_{ \zeta \in W^{1,q'}_0(E),\ \norm{\nb \zeta}_{L^{q'}(E)}=1} \int p \div \zeta$. 
Using $(p)_E=0$,
\[
\int_E |p|^q = \int_E pg = \int_E p \div w \le N \norm{\nb w}_{L^{q'}(E)}\le N C_1\norm{g}_{L^{q'}(E)}
\le NC_1 2 \norm{p}_{L^q(E)}^{q-1}.
\]
Thus $ \norm{p}_{L^q(E)} \le 2NC_1$.
\end{proof}

We next construct a Bogovskii map on an annulus or a half-annulus. It is inspired by \cite[Proposition 2.1]{CPZZ} on a thin disk. See Lemma \ref{CPZZ} for a variation.

\begin{lemma} \label{Bog-annulus}
Let $R>0$, $1<L<\infty$ and $A=B_{LR} \setminus \overline B_R$ or $A=B_{LR}^+ \setminus \overline B_R^+$ be an annulus or a half-annulus in $\R^3$.  There is a linear Bogovskii map $\Bog$ that maps a scalar function $f\in  L^q_0(A)$,  $1<q<\infty$, to a vector field $v=\Bog f\in W^{1,q}_0(A)$ and
\[
\div v = f, \quad \norm{\nb v}_{L^q(A)} \le \frac{C_q}{(L-1)L^{1-1/q}}\, \norm{f}_{L^q(A)}.
\]
The constant $C_q$ is independent of $L$ and $R$.
\end{lemma}

For our applications, $1<L\le 2$ and the constant becomes $\dis \frac{C_q}{L-1}$. The construction below is the same for an annulus or a half-annulus.

\begin{proof} Since the Bogovskii map can be defined by rescaling with the same bound, we may assume $R=1$.
We use  spherical coordinates $\rho, \phi, \th$ with %
$x=\rho(\sin \phi \cos \th, \sin \phi \sin \th, \cos \phi)$, and
write a vector field $v$ as %
\[
v=v_\rho(\rho, \phi, \th) e_\rho + v_\phi(\rho, \phi, \th) e_\phi + v_\th(\rho, \phi, \th) e_\th .
\]
Recall in spherical coordinates,
\[
\div v = \frac1{\rho^2} \pd_\rho (\rho^2 v_\rho) +  \frac1{\rho\sin \phi} \pd_\phi (\sin \phi \, v_\phi) 
+  \frac1{\rho\sin \phi} \pd_\th v_\th .
\]
For given $L\in (1,\infty)$, define $a\in (0,\infty)$ by
\[
L^2 = 3 a^2+1.
\]
We define a new radial variable $\tau \in [1,2]$ by
\[
\tau = \frac 1a \sqrt{\rho^2 + a^2-1},\quad a^2\tau^2 = \rho^2+a^2-1, \quad \frac {d\tau}{d\rho} = \frac{\rho}{a^2\tau}.
\]
It is increasing in $\rho \in [1,L]$, $\tau(\rho=1)=1$ and $\tau(\rho=L)=2$. 

Let $A_0=B_2 \setminus \overline B_1$ if $A=B_{L} \setminus \overline B_1$, or $A_0=B_2^+ \setminus \overline B_1^+$ if  $A=B_{L}^+ \setminus \overline B_1^+$.
For $f(\rho, \phi, \th)$ defined on $A$, we define $\bar f(\tau,\phi,\th)$ on $A_0$ by
\[
\bar f(\tau,\phi,\th) = f(\rho, \phi, \th).
\]
We have
\EQN{
\int_{A_0} |\bar f|^q 
&= \int _1^{2}  \int _0^{2\pi} \int _0^{\pi} |\bar f(\tau,\phi,\th)|^q \tau^2 \sin \phi\, d\phi\, d\th  \, d\tau
\\
&= \int _1^{L}  \int _0^{2\pi} \int _0^{\pi} | f(\rho,\phi,\th)|^q \tau^2 \sin \phi\, \frac{\rho}{a^2\tau}\, d\phi\, d\th  \, d\rho.
}
(We replace $\int_0^\pi$ by $\int_0^{\pi/2}$ if $A_0 \subset \R^3_+$.)
Thus
\EQ{\label{eqB1}
\int_{A_0} |\bar f|^q =\int_{A} | f|^q  \frac{\tau}{a^2 \rho }.
}

Fix one Bogovskii map $\Bog_0$ for the domain $A_0$. Let 
\[
\bar v= \Bog_0 \bke{ \frac \rho \tau  \bar f} =\bar v_\tau e_\tau +\bar v_\phi e_\phi + \bar v_\th e_\th ,
\]
and define $v$ in $A$ from $\bar v$ by
\[
v_\rho(\rho, \phi, \th) = \frac {a^2 \tau^2}{\rho^2} \bar v_\tau (\tau, \phi, \th),\quad
v_\phi(\rho, \phi, \th) =  \bar v_\phi (\tau, \phi, \th),\quad
v_\th(\rho, \phi, \th) = \bar v_\th (\tau, \phi, \th).
\]
We have
\EQN{
\div v &= \frac1{\rho^2} \pd_\rho (\rho^2 v_\rho) +  \frac1{\rho\sin \phi} \pd_\phi (\sin \phi \, v_\phi) 
+  \frac1{\rho\sin \phi} \pd_\th v_\th 
\\
&= \frac1{\rho^2} \frac{d\tau}{d\rho} \pd_\tau (a^2 \tau^2 \bar v_\tau) +  \frac1{\rho\sin \phi} \pd_\phi (\sin \phi \, \bar v_\phi) 
+  \frac1{\rho\sin \phi} \pd_\th \bar v_\th 
}
Using $\frac {d\tau}{d\rho} = \frac{\rho}{a^2\tau}$, we get
\[
\div v = \frac \tau \rho \div \bar v = \bar f = f.
\]
Thus the composition $f \to \bar f \to \bar v \to v$ gives our desired Bogovskii map.

Note that $\pd_\rho v_* = \frac{\rho}{a^2\tau} \pd_\tau \bar v_*$ for $*=\phi,\th$. By \eqref{eqB1}, if $1<L< 10$, 
\[
\int_{A} |\nb v|^q  \lec a^{2-2q} \int _{A_0} |\nb \bar v|^q 
 \lec a^{2-2q} \int _{A_0} | \bar f|^q 
 \lec a^{-2q} \int _{A} | f|^q .
\]
If $10 \le L < \infty$,
\[
\int_{A} |\nb v|^q  \lec a^{2-2q} L \int _{A_0} |\nb \bar v|^q 
 \lec a^{2-2q}L \int _{A_0} | \bar f|^q 
 \lec a^{-2q} L\int _{A} | f|^q .
\]
These show our desired bound, noting $a^{-2}L^{1/q} = 3(L^2-1)^{-1}L^{1/q}$.
\end{proof}

We next construct a Bogovskii map on a region enclosed by cylinders.

\begin{lemma} \label{Bog-cylinder}
Let $R>0$, $1<L<10$ and $E=\bke{B_{LR}' \setminus \overline B_R'}\times (0,R)$ where $B_R'=\bket{ x' \in \R^2: |x'|<R}$ denotes balls in $\R^2$.  There is a linear Bogovskii map $\Bog$ that maps a scalar function $f\in  L^q_0(E)$,  $1<q<\infty$, to a vector field $v=\Bog f\in W^{1,q}_0(E)$ and
\[
\div v = f, \quad \norm{\nb v}_{L^q(E)} \le \frac{C_q}{L-1}\, \norm{f}_{L^q(E)}.
\]
The constant $C_q$ is independent of $L$ and $R$.
\end{lemma}

The upper bound $L<10$ is only for convenience and can be changed.
The construction below can be carried over to a 2D annulus 
$B_{LR}' \setminus \overline B_R'$.%

\begin{proof} Since the Bogovskii map can be defined by rescaling with the same bound, we may assume $R=1$.
We use cylindrical coordinates $r, \th,z$ with 
$x=(r \cos \th, r \sin \th, z)$, and
write a vector field $v$ as
\[
v=v_r(r, \th, z) e_\rho + v_\th(r, \th, z) e_\th + v_z(r, \th, z) e_z .
\]
Recall in cylindrical coordinates,
\[
\div v = \frac1r v_r + \pd_r v_r + \frac 1r \pd_\th v_\th + \pd_z v_z .
\]
We will define a new radial variable $\tau \in [1,2]$ which is increasing in $ r \in [1,L]$, $\tau( r=1)=1$ and $\tau( r=L)=2$. For $f(r, \th, z)$ defined on $E$, we define $\bar f(\tau,\th,z)$ on $E_0=(B_2' \setminus \overline B_1')\times (0,1)$ by
\[
\bar f(\tau,\th,z) = f(r, \th, z).
\]
We have
\EQ{\label{eqB2}
dE_0 = \tau d\tau\, d\th\, dz =  \frac{ \tau}{r}\frac{d \tau}{dr}\, r\, d r\, d\th\, dz = \frac{ \tau}{r}\frac{d \tau}{dr} dE.
}

Fix one Bogovskii map $\Bog_0$ for the domain $E_0$. Let 
\[
\bar v= \Bog_0 \bke{ \frac  r \tau  \bar f} =\bar v_\tau e_\tau +\bar v_\th e_\th + \bar v_z e_z ,
\]
and define $v$ in $E$ from $\bar v$ by
\[
v_ r(r, \th, z) = A(\tau) \bar v_\tau (\tau, \th, z),\quad
v_\th(r, \th, z) =B(\tau)  \bar v_\th (\tau, \th, z),\quad
v_z(r, \th, z) = C(\tau) \bar v_z (\tau, \th, z).
\]
We want to choose $A,B,C$ suitably so that $\div v = D(\tau) \div \bar v$, i.e.,
\EQN{
\div v &=
( \frac1rA +\frac {d\tau}{dr}\pd_\tau A) \bar v_\tau + \frac {d\tau}{dr} A \pd_\tau \bar v_\tau + B\frac 1r \pd_\th v_\th + C\pd_z v_z 
 \\
&= D \bket{\frac1\tau \bar v_\tau + \pd_\tau \bar v_\tau + \frac 1\tau \pd_\th \bar v_\th + \pd_z \bar v_z}.
}
Thus
\EQ{\label{eqB3}
D = \tau ( \frac1rA +\frac {d\tau}{dr}\pd_\tau A) = \frac {d\tau}{dr} A =  \frac \tau r B = C.
}
The second equality gives
\[
\frac 1r \frac {dr}{d\tau} + \frac 1A \frac {dA}{d\tau} = \frac {1}{\tau}.
\]
Integration gives $\ln r + \ln A = \ln \tau + \ln k$ for some $k>0$, or
$A= \frac{k \tau}r$.
We can then solve the rest of \eqref{eqB3} to get
\[
A= \frac{k \tau}r,  \quad B=k \frac {d\tau}{dr} ,\quad
C=D= \frac{k \tau}r  \frac {d\tau}{dr}.
\]
We now choose for convenience
\[
\tau = 1+ \frac{r-1}{L-1}, \quad \frac {d\tau}{dr}= \frac 1{L-1}, \quad k=L-1,
\]
so that
\[
A= \frac{k \tau}r,\quad
B=1, \quad C=D=\frac{\tau}r .
\]

We get
\[
\div v = \frac \tau  r \div \bar v = \bar f = f.
\]
Thus the composition $f \to \bar f \to \bar v \to v$ gives our desired Bogovskii map.

For $1<L< 10$, $B,C,D$ are of order $O(1)$ while $A=O(k)$. Hence
\[
|\pd _\th v| + |\pd _z v|\lec |\pd _\th \bar v| + |\pd _z \bar v|.
\]
Note that $\pd_ r  = \frac{ 1}{k} \pd_\tau$. Also
note that $ \pd_r (\tau/r) = \frac {2-L}{(L-1)r^2}$. Thus $|\pd_r C|  \le C/k$, $|\pd_r A|  \le C$, and
\[
|\pd_r v| \lec k^{-1} ( |\pd_\tau \bar v| + |\bar v|).
\]
By \eqref{eqB2},
\[
\int_{E} |\nb v|^q  \lec k^{1-q} \int _{E_0} |\nb \bar v|^q + | \bar v|^q 
 \lec k^{1-q} \int _{E_0} | \bar f|^q 
 \lec k^{-q} \int _{E} | f|^q .
\]
This shows our desired bound.
\end{proof}

\section{Whole space}
\label{S3}
In this section $\Om=\R^3$ and we will prove Theorem \ref{th1}. Denote $B_R=B_R(0)\subset \R^3$.

We will use the following interior estimate of \cite{Sverak-Tsai}. See \cite[Theorem 3.8]{Kang}, \cite[Remark IV.4.2]{Galdi} (not in \cite{Galdi93a}) and  \cite[\S2.6]{Tsai-book} for alternative proofs.
The key is that the pressure is not needed on the right side of \eqref{ST-eq1}.
\begin{lemma} \label{ST} 
If $(u,p)$ solves
\EQL{Stokes}{
-\De u + \nb p = \div F, \quad \div u=0
}
in $B_{2R}\subset \R^3$,
then for $1<q<\infty$ and $1\le m \le \infty$,
\EQ{\label{ST-eq1}
\norm{\nb u}_{L^q(B_{R})} + \norm{p - (p)_{B_{R}}}_{L^q(B_{R})}
\le C R^{\frac 3q-\frac 3m -1}\,\norm{u}_{L^{m}(B_{2R})} + C \norm{F}_{L^q(B_{2R})},
}
where $C=C(q,m)$ does not depend on $R$.
\end{lemma}

We extend the above to an annulus.

\begin{lemma} \label{int-annulus}
Let $R>0$,  $1<L\le 2$, and $\si=\frac 18(L-1)$. Denote the annuli in $\R^3$
\[
A_R = B_{(L-2\si)R} \setminus \overline  B_{(1+2\si)R}, \quad 
\hA_R = B_{LR} \setminus \overline  B_R.
\]
If $(u,p)$ solves \eqref{Stokes} in $\hA_R$, then for $1<q<\infty$ 
\EQ{\label{ST-eq2}
\norm{\nb u}_{L^q(A_{R})} + \si \norm{p - (p)_{A_{R}}}_{L^q(A_{R})}
\le \frac C{\si R}\,\norm{u}_{L^{q}(\hA_R)} + C \norm{F}_{L^q(\hA_R)},
}
where $C=C(q)$ is uniform in $R>0$ and $\si \in (0,\frac 18]$.
\end{lemma}

Note that the $\si$ factor appears in both sides of \eqref{ST-eq2}.

\begin{proof}
We may assume $R=1$ by scaling.
There are $N=N(L)$ points $x_j\in A_1$, $j=1,\ldots,N$, such that
\[
A_1 \subset \cup_{j=1}^N B_{\si}(x_j),
\]
and
there is a $\si$-independent upper bound for 
the number of overlapping of $ B_{\si}(x_j)$ for $\si \in (0,\frac18]$. Note $\cup_{j=1}^N B_{2\si}(x_j)\subset\hA_1$.
By Lemma \ref{ST} with given $q>1$, $m=q$, and $R=\si$,
\[
\int_{B_{\si}(x_j)} |\nb u|^q
\lec \int_{B_{2\si}(x_j)} \bke{\frac 1{\si^q} |u|^q + |F|^q}.
\]
Summing in $j$,
\[
\int_{A_1} |\nb u|^q \le \sum_{j=1}^N \int_{B_{\si}(x_j)} |\nb u|^q
\lec  \sum_{j=1}^N  \int_{B_{2\si}(x_j)} \bke{\frac 1{\si^q} |u|^q + |F|^q}
\lec \int_{\hA_1} \bke{\frac 1{\si^q} |u|^q + |F|^q}.
\]
This shows the estimate of $\norm{\nb u}_{L^q}$ in \eqref{ST-eq2}.
Apply Lemma \ref{p-est} to $E=A_1$,
\[
\norm{p-(p)_{A_1}}_{L^q(A_1)} \le C_0   \sup_{ \zeta \in W^{1,q'}_0(A_1),\ \norm{\nb \zeta}_{L^{q'}(A_1)}=1} \int p \div \zeta
\]
where $C_0=2\CB(A_1,q')$. By Lemma \ref{Bog-annulus},  $C_0 \le C/\si$. Using the weak form of \eqref{Stokes},
\EQN{
\norm{p-(p)_{A_1}}_{L^q(A_1)} &\le \frac C\si   \sup_{ \zeta \in W^{1,q'}_0(A_1),\ \norm{\nb \zeta}_{L^{q'}(A_1)}=1} \int (\nb u+F) : \nb \zeta
\\
& \le  \frac C\si \bke{ \norm{\nb u}_{L^q(A_1)} + \norm{F}_{L^q(A_1)}}.
}
This completes the proof of \eqref{ST-eq2}.
\end{proof}

\begin{proof}[Proof of Theorem \ref{th1}]
By the standard regularity theory, the solution $u$ of \eqref{SNS} is smooth, $u\in C^\infty_\loc$. For any scalar function $\phi \in C^\infty_c(\R^3)$, we have the local energy equality
\EQ{\label{LEE}
\int |\nb u|^2\phi = \frac 12 \int| u|^2 \De \phi +  \frac 12 \int |u|^2 u\cdot \nb \phi +  \int (p-c) u\cdot \nb \phi
}
by testing \eqref{SNS} with $u\phi$. Above $c$ is any constant. 
Choosing $\phi = \zeta^2$ in \eqref{LEE}, $\zeta \in C^\infty_c$, we get
\EQ{\label{LEE2}
\int |\nb (u\zeta)|^2 &= \int |u|^2 |\nb \zeta |^2
 +   \int |u|^2 u\zeta\cdot \nb \zeta + 2 \int (p-c) u\zeta \cdot \nb \zeta
 \\
 &=I_1+I_2+I_3.
}

Fix $\Theta \in C^\infty(\R)$, $\Theta(t)=1$ for $t<0$, and $\Theta(t)=0$ for $t>1$. We now let $\zeta(x) = \Theta \bke{ \frac {|x|-(L+2\si)R}{4\si R}}$. Then $\zeta\in C^\infty_c(\R^3)$, $\zeta(x)=1$ for $|x|< R(1+2\si)$, $\zeta(x)=0$ for $|x|> R(1+6\si)$, and $|\nb^k \zeta| \le C(\si R)^{-k}$ for $k \in \NN$.
We have
\EQN{
|I_1|\lec (\si R)^{-2} \int_{A_R} |u|^2 \lec
 (\si R)^{-2} |A_R|^{1-2/q}  \bke{\int_{A_R} |u|^q}^{2/q}
\lec
 \bke{\si^{-\frac 12-\frac1{q}} R^{\frac 12-\frac3{q}} \norm{ u}_{L^{q}(\hA_R)} }^2.
}

By H\"older and Sobolev inequalities, for $0 \le \de \le 1$,
\EQN{
|I_2+I_3| &\le C\norm{\nb \zeta}_\infty \cdot \norm{u\zeta}_{6}^\de \cdot \norm{|u|^{3-\de}+|p-c|\cdot|u|^{1-\de}}_{\frac1{1-\de/6},A_R}\\
&\le C (\si R)^{-1} \norm{\nb(u\zeta)}_{2}^\de 
\cdot \bke{\norm{u}_{q}^2+  \norm{p-c}_{q/2}}
\cdot \norm{u}_{q,A_R}^{1-\de}
}
where $q=q(\de)=\frac{3-\de}{1-\de/6}$. By considering \eqref{SNS} as \eqref{Stokes} with $F=- u \otimes u$, we get from Lemma \ref{int-annulus} with $q$ replaced by $q/2$ that
\EQ{\label{eqHolder}
 \norm{p - (p)_{A_{R}}}_{L^{q/2}(A_{R})}
&\le \frac C{\si^2 R}\,\norm{u}_{L^{q/2}(\hA_R)} + \frac C{\si } \norm{|u|^2}_{L^{q/2}(\hA_R)}\\
&\le \frac C{\si^2 R}\,(\si R^3)^{1/q}\norm{u}_{L^{q}(\hA_R)} + \frac C{\si } \norm{u}_{L^{q}(\hA_R)}^2.
}
Thus, choosing $c= (p)_{A_{R}}$,
\EQN{
|I_2+I_3| &\le C (\si R)^{-1} \norm{\nb(u\zeta)}_{2}^\de \cdot
\bke{\si^{\frac 1{q}-2}R^{\frac 3{q}-1} \norm{ u}_{L^{q}(\hA_R)}  + \si^{-1 } \norm{ u}_{L^{q}(\hA_R)}^2 }
\cdot \norm{u}_{q,A_R}^{1-\de}.
}

We now let $\si$ vary and suppose $\si = \si_0 R^{-\al}$, $0\le \al<\infty$, $\si_0>0$. %
Then
\EQ{\label{eq3-6}
|I_1|\lec  \bke{ R^{\frac {1+\al}2+\frac{\al-3}{q}} \norm{ u}_{L^{q}(\hA_R)} }^2,
 \quad
|I_2+I_3| \le C \norm{\nb(u\zeta)}_{2}^\de \cdot J
}
where
\[
J=R^{\frac {3-\al}{q}-2+3\al}\norm{ u}_{L^{q}(\hA_R)}^{2-\de}+ R^{-1+2\al} \norm{ u}_{L^{q}(\hA_R)}^{3-\de}.
\]

By Young's inequality,
\[
|I_2+I_3| \le \frac12 \norm{\nb(u\zeta)}_{2}^2 
+ C J ^{\frac1{1-\de/2}}.
\]
Thus
\EQ{\label{eq3-7}
\int |\nb(u\zeta)|^2
\lec \bke{R^{\frac {1+\al}2+\frac{\al-3}{q}} \norm{ u}_{L^{q}(\hA_R)} }^2
+ J ^{\frac1{1-\de/2}}.
}
To make the right side go to zero, it suffices to find a sequence $R_j\to \infty$, $j\in \NN$, such that
\EQ{\label{eq3-8}
R_j^\be  \norm{ u}_{L^{q}( \hA_{R_j})} \to 0,
}
where $ \hA_{R_j} = \{ x \in \R^3:\  R_j<|x|<R_j(1+8\si_0R_j^{-\al})\}$ and
\[
\be =  \be(\de,\al)=\max \bket{ \frac {1+\al}2+\frac{\al-3}{q},\ \frac{\frac {3-\al}{q}-2+3\al}{2-\de},\ \frac {-1+2\al}{3-\de}}.
\]
Note that the first argument is never greater than the second for $\al \ge0$ and $\de \in [0,1]$, and they equal only if $\al=0$. Thus
\[
 \be(\de,\al)= \max \bket{ %
 \frac{\frac {3-\al}{q}-2+3\al}{2-\de},\ \frac {-1+2\al}{3-\de}}.
\]
If \eqref{eq3-8} is true, then by \eqref{eq3-7},
\[
\lim_{j \to \infty} \int_{|x|<R_j} |\nb u|^2 =0.
\]
Hence $\nb u=0$, $u$ is a constant vector $b$, and
$
R_j^\be  \norm{ u}_{L^{q}( \hA_{R_j})} = C |b| R_j^{\be + (3-\al)/q}
$. Since
$
\be + (3-\al)/q \ge \frac {-1+2\al}{3-\de} +  (3-\al)\frac{1-\de/6}{3-\de} > 0$, 
we get $b=0$ from \eqref{eq3-8}.
This shows both parts (a) and (b), noting that 
 $\be(\de,0)= -\frac 1{3-\de}$.
\end{proof}

\section{Half space}
\label{S4}
In this section we prove Theorem \ref{thmB} for $\Om=\R^3_+$. Its  boundary is $\Ga=\{ (x',0): x'\in \R^2\}$.
For $x_0 \in \Ga$, denote
\[
B_R^+(x_0)= B_R(x_0) \cap \R^3_+, \quad B_R^+= B_R^+(0).
\]
In addition to the interior estimate Lemma \ref{ST},
we will also use the following boundary estimate of Kang \cite{Kang}. Again, the key is that the pressure is not needed on the right side.

\begin{lemma}  \label{Kang}
If $(u,p)$ solves
\EQL{Stokes2}{
-\De u + \nb p = \div F, \quad \div u=0
}
in $B_{\ell R}^+ \subset \R^3_+$, with $u(x',0)=0$, $\ell>1$,
then for $1<q<\infty$ and $1\le m \le \infty$,
\EQ{\label{K-eq1}
\norm{\nb u}_{L^q(B^+_{R})} + \norm{p - (p)_{B^+_{R}}}_{L^q(B^+_{R})}
\le C R^{\frac 3q-\frac 3m -1}\,\norm{u}_{L^{m}(B^+_{\ell R})} + C \norm{F}_{L^q(B^+_{\ell R})},
}
where $C=C(q,m,\ell)$ does not depend on $R$.
\end{lemma}

We extend the above to a half annulus.

\begin{lemma} \label{int-halfannulus}
Let $R>0$,  $1<L\le 2$, and $\si=\frac 18(L-1)$. Denote the half annuli in $\R^3$
\[
A_R = B_{(L-2\si)R}^+ \setminus \overline B_{(1+2\si)R}^+, \quad 
\hA_R = B_{LR}^+ \setminus \overline B_R^+.
\]
If $(u,p)$ solves \eqref{Stokes2} in $\hA_R$, then for $1<q<\infty$ 
\EQ{\label{Kang-eq2}
\norm{\nb u}_{L^q(A_{R})} + \si \norm{p - (p)_{A_{R}}}_{L^q(A_{R})}
\le \frac C{\si R}\,\norm{u}_{L^{q}(\hA_R)} + C \norm{F}_{L^q(\hA_R)},
}
where $C=C(q)$ is uniform in $R>0$ and $\si \in (0,\frac 18]$.
\end{lemma}

Note that the $\si$ factor appears in both sides of \eqref{Kang-eq2}.

\begin{proof}
We may assume $R=1$ by scaling. 
We first choose $N_1=N_1(L)$ points $x^{(j)}\in \overline{ A_1} \cap \Ga$, $j=1,\ldots,N_1$, such that
\[
A_1 \cap\bket{(x',x_3): 0 < x_3< \frac 34\si} \subset \cup_{j=1}^{N_1} B_j, \quad B_j = B_{\si}^+(x^{(j)}).
\]
We then choose $N_2=N_2(L)$ points $x^{(j)}\in A_1$ with $x^{(j)}_3\ge \si$, $j=N_1+1,\ldots,N$ with $N=N_1+N_2$, such that
\[
A_1 \cap\bket{(x',x_3): \frac 34\si\le x_3} \subset \cup_{j=N_1+1}^{N} B_j, \quad B_j = B_{\si/2}(x^{(j)}).
\]
We can choose them in a way that there is a $\si$-independent upper bound for 
the number of overlapping of $ B_j$ for $\si \in (0,\frac18]$. 
We also denote $\widehat B_j =  B_{2\si}^+(x^{(j)})$ for $j\le N_1$ and $\widehat B_j =  B_{\si}(x^{(j)})$ for $j> N_1$. 
Note that $B_j$ and $\widehat B_j$ are half balls for $1\le j\le N_1$ and balls for $N_1+1\le j \le N$. Also note that  $\widehat B_j \subset \hA_1$.
It follows that
\EQL{thmB-eq3}
{
A_1 \subset \cup_{j=1}^N B_j \subset  \cup_{j=1}^N  \widehat B_j  \subset \hA_1.
}

By Lemma \ref{Kang} with given $q>1$, $m=q$, and $R=\si$,
\[
\int_{B_{\si}^+(x^{(j)})} |\nb u|^q
\lec \int_{B^+_{2\si}(x^{(j)})} \bke{\frac 1{\si^q} |u|^q + |F|^q}, \quad (1\le j \le N_1).
\]
By Lemma \ref{ST} with given $q>1$, $m=q$, and $R=\si/2$,
\[
\int_{B_{\si/2}(x^{(j)})} |\nb u|^q
\lec \int_{B_{\si}(x^{(j)})} \bke{\frac 1{\si^q} |u|^q + |F|^q}, \quad (N_1< j \le N).
\]
Summing in $j$,
\[
\int_{A_1} |\nb u|^q \le \sum_{j=1}^N \int_{B_j} |\nb u|^q
\lec  \sum_{j=1}^N  \int_{\widehat B_j } \bke{\frac 1{\si^q} |u|^q + |F|^q}
\lec \int_{\hA_1} \bke{\frac 1{\si^q} |u|^q + |F|^q}.
\]
This shows the estimate of $\norm{\nb u}_{L^q}$ in \eqref{Kang-eq2}.
Apply Lemma \ref{p-est} to $E=A_1$,
\[
\norm{p-(p)_{A_1}}_{L^q(A_1)} \le C_0   \sup_{ \zeta \in W^{1,q'}_0(A_1),\ \norm{\nb \zeta}_{L^{q'}(A_1)}=1} \int p \div \zeta
\]
where $C_0=2\CB(A_1,q')$. By Lemma \ref{Bog-annulus},  $C_0 \le C/\si$. Using the weak form of \eqref{Stokes2},
\EQN{
\norm{p-(p)_{A_1}}_{L^q(A_1)} &\le \frac C\si   \sup_{ \zeta \in W^{1,q'}_0(A_1),\ \norm{\nb \zeta}_{L^{q'}(A_1)}=1} \int (\nb u+F) : \nb \zeta
\\
& \le  \frac C\si \bke{ \norm{\nb u}_{L^q(A_1)} + \norm{F}_{L^q(A_1)}}.
}
This completes the proof of \eqref{Kang-eq2}.
\end{proof}

\begin{proof}[Proof of Theorem \ref{thmB}]
It is the same as the proof of Theorem \ref{th1}, with Lemma \ref{int-annulus}  replaced by Lemma \ref{int-halfannulus}. We use the zero boundary condition when we integrate by parts and when we apply the Sobolev inequality.
\end{proof}

\section{Periodic slab}
\label{S5}
In this section we will prove Theorem \ref{thmC} for the periodic slab $\Om=\R^2\times (\R / \ZZ)$. We can identify our velocity field $v(x)=v(x',x_3)$ as a vector field defined in $\R^3$ that is periodic in $x_3$, and %
the system \eqref{SNS} is satisfied in the whole $\R^3$. %

We first extend the interior estimate, Lemma \ref{int-annulus}, to a region enclosed by cylinders.

\begin{lemma} \label{int-cyl}
Let $R>0$,  $1<L\le 2$, and $\si=\frac 18(L-1)$. Denote the 3D regions
\[
E_R := A_R \times(0,R), \quad \hE_R := \hA_R \times (-2\si R, (1+2\si)R),
\]
where
\[
A_R = B_{(L-2\si)R}' \setminus \overline{ B_{(1+2\si)R}'}, \quad
\hA_R = B_{LR}' \setminus \overline{B_R' }
\]
are 2D annuli, and $B_R'=\bket{ x' \in \R^2: |x'|<R}$.  
If $(u,p)$ solves the Stokes system \eqref{Stokes} in $\hE_R$, then for $1<q<\infty$ 
\EQ{\label{ST-cyl}
\norm{\nb u}_{L^q(E_{R})} + \si \norm{p - (p)_{E_{R}}}_{L^q(E_{R})}
\le \frac C{\si R}\,\norm{u}_{L^{q}(\hE_R)} + C \norm{F}_{L^q(\hE_R)},
}
where $C=C(q)$ is uniform in $R>0$ and $\si \in (0,\frac 18]$.
\end{lemma}

 Its proof is identical to the proof of Lemma \ref{int-annulus}, with the reference to Lemma \ref{Bog-annulus} replaced by Lemma \ref{Bog-cylinder}.

\begin{proof}[Proof of Theorem \ref{thmC}]
By the standard regularity theory, $u\in C^\infty_\loc$.  Using the periodic BC, we still have the local energy equality
\eqref{LEE2} for any scalar function $\zeta \in C^\infty_c( \Om)$ and any constant $c$,
\EQ{\label{LEE3}
\int_\Om |\nb (u\zeta)|^2 &= \int_\Om |u|^2 |\nb \zeta |^2
 +   \int_\Om |u|^2 u\zeta\cdot \nb \zeta + 2 \int_\Om (p-c) u\zeta \cdot \nb \zeta
 \\
 &=I_1+I_2+I_3.
}

By considering \eqref{SNS} as \eqref{Stokes} with $F=- u \otimes u$, we get from Lemma \ref{int-cyl} with $q$ replaced by $q/2$ that
\[
 \norm{p - (p)_{E_{R}}}_{L^{q/2}(E_{R})}
\le \frac C{\si^2 R}\,\norm{u}_{L^{q/2}(\hE_R)} + \frac C{\si } \norm{|u|^2}_{L^{q/2}(\hE_R)}.
\]
Raising to the $q/2$-th power and using the periodicity with $L\le 2 \ll R$,
we get
\EQN{
R\int_0^1 \int_{A_R} |p-(p)_{E_R}|^{q/2}  &\le  \frac {CR}{\si^q R^{q/2}} \int_0^1 \int_{\hA_R}|u|^{q/2} +  \frac {CR}{\si^{q/2} }\int_0^1 \int_{\hA_R}|u|^q.
}
Thus, with $\Om_R:= \hA_R\times(0,1)$,
\EQL{p-ARper-est}{
 \norm{p - (p)_{E_{R}}}_{L^{q/2}(A_{R}\times(0,1))}
&\le \frac C{\si^2 R}\,\norm{u}_{L^{q/2}(\Om_R)} + \frac C{\si } \norm{u}^2_{L^{q}(\Om_R)}
\\
&\le \frac C{\si^2 R}\,(\si R^2)^{1/q}\norm{u}_{L^{q}(\Om_R)} + \frac C{\si } \norm{u}^2_{L^{q}(\Om_R)}.
}
Here we use $|\Om_R|=C \si R^2$, which is different from $|A_R|=C\si R^3$ in the proofs of Theorems \ref{th1} and \ref{thmB}.

Fix $Z \in C^\infty_c(\R^2)$, $Z(x')=1$ for $|x'|< 1+ 2\si$, $Z(x')=0$ for $|x'|>L-2\si$. Thus $\nb Z$ is supported in $\overline{A_1}$. Let $\zeta(x) =\zeta_R(x)= Z(\frac {x'}R)$ in the local energy equality \eqref{LEE3}, with $\nb \zeta_R$ supported in $\overline{A_R}\times (\R /\ZZ)$.  Note $\zeta$ does not depend on $x_3$. Choose $c= (p)_{E_{R}}$.
By H\"older and Sobolev inequalities, for $0 \le \de \le 1$ and $q=q(\de)=\frac{3-\de}{1-\de/6}$,
\EQN{
|I_2+I_3| &\le C\norm{\nb \zeta}_\infty \cdot \norm{u\zeta}_{6}^\de \cdot \norm{|u|^{3-\de}+|p-c|\cdot|u|^{1-\de}}_{\frac1{1-\de/6},A_R
\times(0,1)}\\
&\le C (\si R)^{-1} \norm{\nb(u\zeta)}_{2}^\de 
\cdot \bke{\norm{u}_{q}^2+  \norm{p-c}_{q/2,\, A_R\times(0,1)}}
\cdot \norm{u}_{q,\, A_R\times(0,1)}^{1-\de}.
}
By \eqref{p-ARper-est},
\[
|I_2+I_3|
 \le C (\si R)^{-1} \norm{\nb(u\zeta)}_{2}^\de \cdot
\bke{\frac 1{\si^2 R}\,(\si R^2)^{1/q}\norm{u}_{L^{q}(\Om_R)} + \frac 1{\si } \norm{u}^2_{L^{q}(\Om_R)}}
\cdot \norm{u}_{q,\Om_R}^{1-\de}.
\]

We now let $\si$ vary and suppose $\si = \si_0 R^{-\al}$, $0\le \al<\infty$, $\si_0>0$.
Then
we have
\EQN{
|I_1| &\lec (\si R)^{-2} \int_{\Om_R} |u|^2 \lec
 (\si R)^{-2} (\si R^2)^{1-2/q} \norm{u}_{q,\Om_R}^2\\
&=
 \bke{\si^{-\frac 12-\frac1{q}} R^{-\frac2{q}} \norm{ u}_{L^{q}(\Om_R)} }^2
=C\bke{ R^{\frac \al 2+\frac{\al-2}{q}} \norm{ u}_{L^{q}(\Om_R)} }^2.
}
We also have by Young's inequality,
\EQN{
|I_2+I_3| &\le C \norm{\nb(u\zeta)}_{2}^\de \cdot J  \le \frac12 \norm{\nb(u\zeta)}_{2}^2 
+ C J ^{\frac1{1-\de/2}},
\\
J&=R^{\frac {2-\al}{q}-2+3\al}\norm{ u}_{L^{q}(\Om_R)}^{2-\de}+ R^{-1+2\al} \norm{ u}_{L^{q}(\Om_R)}^{3-\de}.
}
Thus
\EQ{\label{eq5-4}
\int |\nb(u\zeta)|^2
\lec \bke{ R^{\frac \al 2+\frac{\al-2}{q}} \norm{ u}_{L^{q}(\Om_R)} }^2
+ J ^{\frac1{1-\de/2}}.
}
To make the right side go to zero, it suffices to find a sequence $R_j\to \infty$, $j\in \NN$, such that
\EQ{\label{eq5-8}
R_j^\be  \norm{ u}_{L^{q}( \Om_{R_j})} \to 0,
}
where $ \Om_{R_j} = \{ x=(x',x_3) \in \Om:\  R_j<|x'|<R_j(1+8\si_0R_j^{-\al})\}$ and
\[
\be = \be_{ps}(\de,\al)=\max \bket{ \frac {\al}2+\frac{\al-2}{q},\ \frac{\frac {2-\al}{q}-2+3\al}{2-\de},\ \frac {-1+2\al}{3-\de}}=\max \{\be_1,\be_2,\be_3\}.
\]
Numerically $\be_1 \le\be_3$ for $\al \le 2$, and $\be_1 \le\be_2$ for $\al \ge 0.2$. Thus we can drop $\be_1$,
\[
\be_{ps}(\de,\al)=\max \bket{ \frac{\frac {2-\al}{q}-2+3\al}{2-\de},\ \frac {-1+2\al}{3-\de}}.
\]

If \eqref{eq5-8} is true, then by \eqref{eq5-4},
\[
\lim_{j \to \infty} \int_{|x'|<R_j} |\nb u|^2 =0.
\]
Hence $\nb u=0$, $u$ is a constant vector $b$, and
$
R_j^\be  \norm{ u}_{L^{q}( \Om_{R_j})} = C |b| R_j^{\be + (2-\al)/q}
$. Since
$
\be + (2-\al)/q \ge \frac {-1+2\al}{3-\de} +  (2-\al)\frac{1-\de/6}{3-\de} > 0$, 
we get $b=0$ from \eqref{eq5-8}.
This shows both parts (a) and (b), noting that $\be_{ps}(\de,0)= -\frac {1}{3-\de}$.
\end{proof}

\section{Zero BC slab}
\label{S6}
In this section we prove Theorem \ref{th4}. The domain is $\Om=\R^2\times (0,1)$ and the vector field $u$ satisfies the zero boundary condition $u(x',0)=u(x',1)=0$. Its proof is different from those for Theorems \ref{th1}-\ref{thmC} as we cannot obtain the local pressure estimate by scaling. We also cannot vary the radii ratio $L$.

Denote $B_R'=\{ x' \in \R^2: \ |x'|<R\}$ and $B_R'(x_0')=\{ x' \in \R^2: \ |x'-x_0'|<R\}$.
For given $L>1$, let $\si=\frac 18(L-1)$. Denote
the 2-D annuli
\[
A_R = B_{(L-2\si)R}' \setminus B_{(1+2\si)R}', \quad
\hA_R = B_{LR}' \setminus B_R' .
\]

We will use a variation of \cite[Proposition 2.1]{CPZZ}:

\begin{lemma}\label{CPZZ} There is a (linear) \emph{Bogovskii map} $\Bog_R$ on $E=A_R \times (0,1)$, $R \ge 1$, that maps  $f\in L^q_0(E)=\{ f \in L^q(E):\ \int_E f =0\}$ to $v = \Bog_R f \in W^{1,q}_0(E;\R^3)$ satisfying
\[
\div v = f, \quad 
\norm{\nb v}_{L^q(E)} \le C_{q,L} R \norm{f}_{L^q(E)}.
\]
\end{lemma}
The key is that the constant grows linearly in $R$.
The proof is the same as \cite[Proposition 2.1]{CPZZ} which is formulated for $E=B_R'\times(0,1)$ and $q=2$. The idea is to fix $\Bog_1$ for $E_1=A_1\times (0,1)$, and for given $f(x) \in L^q_0(E)$, define $\bar f(\bar x) = f(x)$  for $\bar x \in E_1$ with $x=(R\bar x_1, R\bar x_2, \bar x_3)$, let $\bar v = \Bog_1 \bar f$, and then $v=\Bog_R f$ is defined by $v(x) = (R\bar v_1(\bar x), R\bar v_2(\bar x), \bar v_3(\bar x))$. %

\begin{proof}[Proof of Theorem \ref{th4}]
We may assume $L\le 2$ and let $A_R$ and $\hA_R$ be defined as above.
By the standard regularity theory, $u\in C^\infty_\loc(\overline \Om)$.  We still have the local energy equality
\eqref{LEE} for any scalar function $\phi \in C^\infty_c(\Om)$.

There are $N_0=N_0(L)\in \NN$ and $R_0=R_0(L)\gg 1$ such that, if $R \ge R_0$, then there are $N\le N_0 R^2$ points $x_j\in A_R$, $j=1,\ldots,N$, such that
\EQL{th4-eq3}{
A_R \subset \cup_{j=1}^N B_{1/2}'(x_j) \subset \cup_{j=1}^N B_{1}'(x_j) \subset \hA_R,
}
and there is an $R$-independent upper bound for
the number of overlapping of $ B_j'(x_j)$. 
Unlike in the proofs of the previous theorems, the choice of $x_j$ is not by scaling, and the radii do not depend on $L$.

We can rewrite \eqref{SNS} as \eqref{Stokes} with $F=-u\otimes u$.
By Lemma \ref{Kang} and zero BC at $x_3=0$, for $q=m=r>1$ and $\ell=\sqrt2$,
\[
\norm{\nb u}_{L^r(B_{1/2}'(x_j)  \times (0,\frac 12))}^r \lec \norm{ u}_{L^{r}(B_{1}'(x_j)  \times (0,1))}^r  + \norm{ u}_{L^{2r}(B_{1}'(x_j)  \times (0,1))}^{2r} .
\]
In applying Lemma \ref{Kang} we have used the inclusions
\[
B_{\frac 12}'(x_j)  \times (0,\tfrac 12)
\subset B^+_{\sqrt 2/2} ((x_j,0))
\subset B^+_{1} ((x_j,0))
\subset B_{1}'(x_j)  \times (0,1).
\]
Similarly, by Lemma \ref{Kang} and zero BC at $x_3=1$,
\[
\norm{\nb u}_{L^r(B_{1/2}'(x_j)  \times (\frac12,1))}^r \lec \norm{ u}_{L^{r}(B_{1}'(x_j)  \times (0,1))}^r  + \norm{ u}_{L^{2r}(B_{1}'(x_j)  \times (0,1))}^{2r} .
\]
Denote $E_R=A_R\times (0,1)$ and $\Om_R=\hA_R\times (0,1)$.
Summing the two estimates, summing in $j$ and using \eqref{th4-eq3}, we get
\EQL{th4-er4}{
\norm{\nb u}_{L^r(E_R)}^r \lec \norm{ u}_{L^{r}( \Om_R)}^r  + \norm{ u}_{L^{2r}(\Om_R)}^{2r} .
}
By Lemma \ref{p-est} and Lemma \ref{CPZZ} with $p_R =\frac 1{|A_R|}\int_0^1\int_{A_R} p$,
\[
\norm{p-p_R}_{L^r(E_R)}
\le CR  \sup_{ \zeta \in W^{1,r'}_0(E_R),\ \norm{\nb \zeta}_{L^{r'}(E_R)}=1} \int p \div \zeta.
\]
Using weak form of Stokes system \eqref{Stokes}, and then
\eqref{th4-er4},
\EQL{th4-eq5}{
\norm{p-p_R}_{L^r(E_R)}
&\lec R
\bke{\norm{\nb u}_{L^r(E_R)} + \norm{u\otimes u}_{L^r(E_R)} }
 \lec R \bke{\norm{ u}_{L^{r}(\Om_R )}  + \norm{ u}_{L^{2r}(\Om_R)}^{2}} .
}

By the same argument between \eqref{p-ARper-est} and \eqref{eq5-4} and same $\zeta_R$, with  $\al=0$ and $\si$ ignored being a constant, we get for $r>1$,
\EQ{\label{eq6.4}
\int_\Om |\nb (u\zeta)|^2 \lec \bke{ R^{-1/r} \norm{u}_{L^{2r}(\Om_R)}}^2 +  J ^{\frac1{1-\de/2}}
}
where
\[
J= \bke{ \norm{ u}_{L^{r}(\Om_R )}  + \norm{ u}_{L^{2r}(\Om_R)}^{2} }  \norm{ u}_{L^{s}(\Om_R )}^{1-\de},
\quad \frac{1-\de}s = 1-\frac \de 6 - \frac 1r.
\]
Note that the factor $R^{-1}$ from $|\nb \zeta|$ cancels the factor $R$ in \eqref{th4-eq5}.

If we set $s=2r$, then $s=2r=q(\de)=\frac{3-\de}{1-\de/6}$. Using $\norm{u}_{L^{q/2}(\Om_R )}  \lec R^{2/q} \norm{ u}_{L^{q}(\Om_R )}$, \eqref{eq6.4} becomes
\[
\int_\Om |\nb (u\zeta)|^2 \lec \bke{ R^{-2/q} \norm{u}_{L^q(\Om_R)}}^2 + \bke{R^{2/q} \norm{ u}^{2-\de}_{L^{q}(\Om_R )}  + \norm{ u}_{L^{q}(\Om_R)}^{3-\de}} ^{\frac1{1-\de/2}}.
\]
Thus, if condition \eqref{th4eq1} holds, there is a sequence $R_j \to \infty$ as $j \to \infty$ such that $R_j^{2/q} \norm{ u}^{2-\de}_{L^{q}(\Om_{R_j} )}\to 0$, then the above shows  $\lim_{j \to \infty}\int_{|x'|<R_j} |\nb u|^2=0$, and hence $u \equiv0$.

Alternatively, if condition \eqref{th4eq1b} holds that there is a sequence $R_j \to \infty$ as $j \to \infty$ such that $ \norm{ u}_{L^{r}(\Om_{R_j}  )}  + \norm{ u}_{L^{2r}(\Om_{R_j})} \to 0$, and suppose $r \le s \le 2r$, then $J \to 0$, $\lim_{j \to \infty}\int_{|x'|<R_j} |\nb u|^2=0$, and hence $u \equiv0$. The condition $r \le s \le 2r$ is equivalent to 
\[
r_1(\de) = \frac{3(3-\de)}{6-\de} \le r \le r_2(\de) = \frac{6(2-\de)}{6-\de}.
\]
Both $r_1,r_2$ are decreasing in $\de$ with $r_1(0)=3/2$, $r_2(0)=2$, and $r_1(1)=r_2(1)=6/5$. Hence for any $r \in [6/5,2]$, we can find $\de \in [0,1]$ such that $r_1(\de) \le r \le r_2(\de)$. This shows part (b).
\end{proof}

\end{document}